\documentclass[11pt, a4paper]{amsart}

\usepackage{amsmath}
\usepackage{amssymb}
\usepackage{graphicx}
\usepackage{hyperref, enumerate}

\newtheorem{theorem}{Theorem}[section]

\newtheorem{lemma}[theorem]{Lemma}
\newtheorem{proposition}[theorem]{Proposition}
\theoremstyle{definition}

\newtheorem{remark}[theorem]{Remark}

\numberwithin{equation}{section}

\DeclareMathOperator*{\Ran}{Ran}
\DeclareMathOperator*{\Ker}{Ker}
\DeclareMathOperator*{\Fix}{Fix}
\DeclareMathOperator*{\dist}{dist}
\DeclareMathOperator*{\R}{Re}

\newcommand{\B}{\mathcal{B}}
\newcommand{\T}{\mathbb{T}}

\newcommand{\CC}{\mathbb{C}}

\newcommand{\DD}{\mathbb{D}}
\newcommand{\ep}{\varepsilon}

\title[Quantified asymptotic behaviour of Banach space operators]{Quantified asymptotic behaviour of Banach space operators and applications to iterative projection methods}

\author[C. Badea]{Catalin Badea}
\address[C. Badea]{Universit\'e de Lille, CNRS UMR 8524, Laboratoire Paul Painlev\'e, 59655 Villeneuve d'Ascq, France}
\email{\tt catalin.badea@univ-lille.fr}

\author[D. Seifert]{David Seifert}
\address[D. Seifert]{St John's College, St Giles, Oxford\;\;OX1 3JP, United Kingdom}
\email{\tt david.seifert@sjc.ox.ac.uk}

\thanks{Acknowledgements. This work was supported in part by the Labex CEMPI (ANR-11-LABX-0007-01).}

\begin{document}

\begin{abstract}
We present an extension of our earlier work [Ritt operators and convergence in the method of alternating
  projections, {\em J. Approx. Theory}, 205:133--148, 2016] by proving a general asymptotic result for orbits of an operator acting on a reflexive  Banach space. This result is obtained under a condition involving the growth of the resolvent, and we also discuss conditions involving the location and the geometry of the numerical range of the operator. We then apply the general results to some classes of iterative projection methods in approximation theory, such as the Douglas-Rachford splitting method and, under suitable  geometric conditions either on the ambient Banach space or on the projection operators,  the method of alternating projections. 
\end{abstract}

\maketitle

%%%%%%%%%%%%%%%%%%%%%%%%%%%%%%%%%%%%%%%%%%%%%%%%%%%%%%%%%%%%%%%%%%%%%%%%%%%%%%%%

\section{Introduction}\label{sec:intro} 

Many problems in approximation theory can be formulated in terms of operators acting on suitable Banach or Hilbert spaces. Two particularly simple yet powerful methods in approximation theory are the Douglas-Rachford splitting method and the method of alternating projections (for two or more subspaces). These two iterative projection methods play important roles in convex optimisation, differential equations and signal processing; see for instance the references in \cite{BDNPW16}. In both cases one is led to consider a bounded linear operator $T$ acting on a Hilbert space $X$, and one is interested in the asymptotic behaviour of the orbits $(T^nx)_{n\ge0}$ as $n\to\infty$ for different initial vectors $x\in X$. It is now well known (see for instance the survey \cite{DeHuSurvey1}) that in situations such as these one expects a dichotomy for the rate of convergence: either the sequence $(T^nx)_{n\ge0}$ converges exponentially fast  for all $x\in X$ or it converges arbitrarily slowly for suitably chosen initial vectors $x\in X$. It was recently shown by the authors in \cite{BaSe16} that one can say more in the case of the method of alternating projections, namely that even when the convergence is arbitrarily slow there exists a rich supply of initial vectors $x\in X$ for which the sequence $(T^nx)_{n\ge0}$ converges to a limit at a rate faster than any polynomial rate. The purpose of this paper is to present a general result in the Banach space situation which extends the approach in \cite{BaSe16} so as to be applicable to a wider class of iterative schemes, including the Douglas-Rachford splitting method, and to more general ambient Banach spaces.

The paper is organised as follows. In Section~\ref{sec:gen} we present a general result, Theorem~\ref{thm:gen},  describing the rate of convergence of orbits of certain operators acting on a Banach space under an assumption involving the growth of the resolvent. Then in Section~\ref{sec:numran} we obtain conditions on the location and the geometry of the numerical range of the operator which ensure that the required resolvent growth condition in Theorem~\ref{thm:gen} is satisfied. 
Finally, in Sections~\ref{sec:appl1} and \ref{sec:appl2} we show how the general theory can be applied to particular methods in approximation theory, namely the Douglas-Rachford splitting method and the method of alternating projections. For the method of alternating projections we discuss, in particular, products of orthoprojections in uniformly convex and uniformly smooth Banach spaces, while for the Douglas-Rachford splitting method we discuss a variant of the original result for several subspaces of a given reflexive Banach space. 

The notation we use is standard. Given a Banach space $X$,  assumed to be complex throughout, we write $\B(X)$ for the algebra of bounded linear operators $T\colon X\to X$. The identity operator on $X$ is denoted by $I_X$, or simply by $I$ if the space $X$ is clear from the context. Given an operator $T\in \B(X)$ we write $\Ker T$ for the kernel of $T$ and $\Ran T$ for the range of $T$, and we let $\Fix T=\Ker (I-T)$. An operator $T$ is said to be power-bounded if $\sup_{n\ge1}\|T^n\| < \infty$. Moreover, given $T\in\B(X)$ we let $\sigma(T)$ denote the spectrum of $T$, $r(T)$ its spectral radius and we let $R(\lambda,T)$ be the resolvent operator $(\lambda I-T)^{-1}$ when $\lambda\in\CC\setminus\sigma(T)$. The dual of $X$ is denoted by $X^{\ast}$ and we write  $\langle x,\phi\rangle =\phi(x)$ for $\phi\in X^{\ast}$ and $x\in X$. If $X$ is a Hilbert space and $T\in\B(X)$ we write $W(T)=\{(Tx,x):x\in X,\|x\|=1\}$ for the numerical range (field of values) of $T$. An extension of the numerical range to the Banach space situation will be introduced in Section~\ref{sec:numran}. Furthermore, we let $\DD=\{\lambda\in\CC:|\lambda|<1\}$ and  $\T=\{\lambda\in\CC:|\lambda|=1\}$, and we use `big O' and `little o' notation in the usual way.  Other notation and definitions will be introduced when needed. 

\section{A general dichotomy result for the rate of convergence}\label{sec:gen}

Let $X$ be a Banach space and suppose that $T_n\in\B(X)$, $n\ge0$, are operators such that $\|T_nx\|\to0$ as $n\to\infty$ for all $x\in X$. We say that the convergence is \emph{arbitrarily slow} if for every sequence $(r_n)_{n\ge0}$ of non-negative scalars satisfying $r_n\to0$ as $n\to\infty$ there exists $x\in X$ such that $\|T_nx\|\ge r_n$ for all $n\ge0$. We say that the convergence is \emph{weakly arbitrarily slow} if for every sequence $(r_n)_{n\ge0}$ of non-negative scalars satisfying $r_n\to0$ as $n\to\infty$ there exist $x\in X$ and $\phi\in X^{\ast}$ such that $\R \langle T_nx,\phi\rangle \ge r_n$ for all $n\ge0$.  
We say that the convergence is \emph{superpolynomially fast} for a particular $x\in X$ if $\|T_n x\|=o(n^{-k})$ as $n\to\infty$ for all $k\ge1$. \medskip

\begin{theorem}\label{thm:gen}
Let $X$ be a reflexive Banach space and suppose that $T\in\B(X)$ is a power-bounded operator such that $\sigma(T)\cap\T\subseteq\{1\}$ and, for some $\alpha\ge1$, 
\begin{equation}\label{eq:res}
\|R(e^{i\theta},T)\|=O(|\theta|^{-\alpha}),\quad \theta\to0.
\end{equation}
Then $X=\Fix T\oplus Z$, where $Z$ denotes the closure of $\Ran(I-T)$, and for all $x\in X$ we have
\begin{equation}\label{eq:conv}
\|T^nx-P_Tx\|\to0,\quad n\to\infty,
\end{equation}
where $P_T$ denotes the projection onto $\Fix T$ along $Z$. 
Moreover, there is a dichotomy for the rate of convergence. Indeed, if $\Ran(I-T)$ is closed then there exist $C>0$ and $r\in[0,1)$ such that
\begin{equation}\label{eq:exp}
\|T^n-P_T\|\le C r^n, \quad n\ge0,
\end{equation}
whereas if $\Ran(I-T)$ is not closed then the convergence in \eqref{eq:conv} is arbitrarily slow and weakly arbitrarily slow. In either case there exists a dense subspace $X_0$ of $X$ such that for all $x\in X_0$ the convergence in \eqref{eq:conv} is superpolynomially fast.
\end{theorem}

\begin{proof}
Since $T$ is assumed to be power-bounded and $X$ is reflexive it follows from classical ergodic theory that $X=\Fix T\oplus Z$; see for instance \cite[Section~2.1]{Kre85}. Let $Y=\Ran (I-T)$. Since $T$ is power-bounded and $\sigma(T)\cap\T\subseteq\{1\}$, it follows from the Katznelson-Tzafriri theorem \cite[Theorem~1]{KT86} that $\|T^nx\|\to0$ as $n\to\infty$ for all $x\in Y$. By a simple density argument the same is true for all $x\in Z$. Since $T^nx=x$ for all $x\in \Fix T$ and $n\ge0$, we may deduce \eqref{eq:conv}. 

Note that both of the spaces $\Fix T$ and $Z$ are invariant under $T$. Thus if we let $S$ denote the restriction of $T$ to $Z$, then it is easy to see that $\sigma(S)\subseteq\sigma(T)\subseteq\DD\cup\{1\}$, and in particular $\sigma(S)\cap\T\subseteq\{1\}$. Note also that $I_Z-S$ maps $Z$ bijectively onto $Y$. It follows from the Inverse Mapping Theorem that $1\in\sigma(S)$ if and only if $Y\ne Z$, which is to say if and only if $Y$ is not closed. Thus if $Y$ is closed then $r(S)<1$ and we may find, for each $r\in (r(S),1)$ a suitable constant $C>0$ such that 
$$\|T^n-P_T\|\le \|S^n\|\|I-P_T\|\le Cr^n,\quad n\ge0,$$
so \eqref{eq:exp} holds. On the other hand, if $Y$ is not closed then $r(S)=1$ and it follows from \cite[Theorem~1]{Mue88} that the convergence in \eqref{eq:conv} is arbitrarily slow. Furthermore, the space $X$, being reflexive, does not contain an isomorphic copy of $c_0$, so it follows from \cite[Theorem~1]{Mue05} that the convergence in \eqref{eq:conv} is weakly arbitrarily slow. Now using   \cite[Theorem~2.5]{Sei15b} or \cite[Theorem~2.11]{Sei16} it follows from  assumption \eqref{eq:res} that
\begin{equation}\label{eq:rate}
\|S^nx\|=O\bigg(\frac{(\log n)^{1/\alpha}}{n^{1/\alpha}}\bigg),\quad n\to\infty,
\end{equation}
for all $x\in Y$. If we let $Y_k=\Ran(I-T)^k$, $k\ge1$, then iterating the estimate in \eqref{eq:rate} shows that for  $x\in Y_k$, $k\ge1$, we have
$$\|S^nx\|=O\bigg(\frac{(\log n)^{k/\alpha}}{n^{k/\alpha}}\bigg),\quad n\to\infty.$$
 Let $X_k=\Fix T\oplus Y_k$, $k\ge1$. Then each $X_k$ is dense in $X$ and, for $k\ge1$ and $x\in X_k$, we have 
$$\|T^nx-P_Tx\|=O\bigg(\frac{(\log n)^{k/\alpha}}{n^{k/\alpha}}\bigg),\quad n\to\infty.$$
Now let $X_0=\bigcap_{k=1}^\infty X_k$. It follows from the Esterle-Mittag-Leffler theorem \cite[Theorem~2.1]{Est84} that $X_0$ is a dense subspace of $X$, and it is clear that for all $x\in X_0$ the convergence in \eqref{eq:conv} is superpolynomially fast.
\end{proof}

\begin{remark}
 \label{rem:Hilbert}
Given a Banach space $X$ and an operator $T\in \B(X)$ with $r(T)\le 1$, the resolvent condition \eqref{eq:res} is equivalent to having $
\|R(\lambda,T)\|=O(|\lambda - 1|^{-\alpha})$ as $\lambda\to1$ with  $|\lambda| > 1$; see  \cite[Lemma~3.3]{CoLi16} and \cite[Lemma~3.9]{Sei16}. If these equivalent conditions are satisfied for $\alpha=1$ then $T$ is said to be a \emph{Ritt operator}. It was shown in \cite{Ly99, NaZe99} that $T$ is a Ritt operator if and only if $T$ is power-bounded and satisfies $\|T^n(I-T)\|=O(n^{-1})$ as $n\to\infty$. Thus for the case of Ritt operators the proof of Theorem~\ref{thm:gen} can be simplified, and in particular the logarithmic terms are not needed. By \cite[Theorem~3.10]{Sei16} the logarithmic factors are also redundant if $X$ is a Hilbert space.
\end{remark}

\section{The numerical range and generalised Stolz domains}\label{sec:numran}

Theorem~\ref{thm:gen} gives a detailed description of the asymptotic behaviour of orbits of certain bounded linear operators $T$ in terms of the growth of the resolvent. In general, when estimating the growth of the resolvent one needs precise spectral information about $T$. One important case in which such information is available is when the location and the geometry of the numerical range of $T$ are known.  

We begin by recalling the notion of numerical range for a bounded linear operator $T$ acting on a Banach space $X$; see \cite{BoDuNR} for more information. Let $J\colon X\to X^*$  be an isometry with the property that $\langle x, \phi_x \rangle  
= \|x\|^2$ for all $x\in X$, where $\phi_x=J(x)$; note that the existence of such maps  is a straightforward consequence of the Hahn-Banach theorem.  We define the \emph{numerical range} of $T\in \B(X)$ as 
$$ W(T) = \{\langle Tx,\phi_x\rangle : \|x\| = 1\}.$$
Although the numerical range of $T$ depends on the choice of the map $J$, its closed convex hull $\mathrm{co}\,W(T)$ does not.  Indeed, according to \cite[Theorem~6]{Tohoku} (see also \cite{BoDuNR}) we have $ \mathrm{co}\, W(T) = W_0(T),$ where 
$$W_0(T) = \{\phi(T) :  \phi\in \B(X)^{\ast}, \|\phi\| = \phi(I_{X}) = 1\}.$$
If $X$ is a Hilbert space then $W_0(T)$ coincides with the closure of the usual numerical range $W(T)$ of $T$; see \cite[page~420]{Tohoku}.  

Our main aim in this section is to obtain a version of Theorem~\ref{thm:gen} under conditions involving geometric assumptions on the numerical range of the operator $T$. We say that a non-empty closed set $\Omega\subseteq\DD\cup\{1\}$ is a \emph{generalised Stolz domain} if there exist constants $c ,\varepsilon > 0$ and $\alpha\ge1$ such that 
\begin{equation}\label{eq:stolz}
1-|\lambda|\ge c|\lambda-1|^{\alpha}
\end{equation}
for all $\lambda\in \Omega$ with  $|\lambda-1|\le \varepsilon$. In particular, any closed subset of a generalised Stolz domain is itself a generalised Stolz domain. If \eqref{eq:stolz} holds for a set $\Omega$ then  we say that $\Omega$ is a \emph{generalised Stolz domain with (Stolz) parameter $\alpha$}. Note that if $\Omega$ is a generalised Stolz domain with parameter $\alpha$, then  $\Omega$ is also a generalised Stolz domain with parameter $\beta$ for any $\beta\ge\alpha$. If $\Omega$ is the convex hull of the set $\{\lambda\in\CC:|\lambda|\le r\}\cup\{1\}$ for some $r\in(0,1)$, then $\Omega$ is said to be a \emph{Stolz domain}; note, however, that there is some inconsistency in the use of this terminology throughout the literature. Any Stolz domain is a  generalised Stolz domain with parameter $\alpha=1$.  Other important examples of generalised Stolz domains are sets of the form $\{\lambda\in\CC:|\lambda-r|\le 1-r\}$ for some $r\in (0,1)$. In the language of hyperbolic geometry such sets (or, more precisely, their boundaries) are examples of \emph{horocycles}, and they are generalised Stolz domains with parameter $2$. Another class of regions closely related to generalised Stolz domains are so-called \emph{quasi-Stolz} domains; see \cite{CoLi16, Pau12} for details.

\begin{remark}
If $\Omega$ is a Stolz domain with parameter $\alpha\ge1$, then in the terminology of \cite{SpSt96,SpSt97,SpSt98} we have that $1$ is a point of contact of type $\alpha-1$ between $\Omega$ and the unit circle $\T$. More general subsets of the closed unit disc, even ones touching the unit circle in a finite number of points, have been considered in \cite{SpSt96,SpSt97,SpSt98} in relation to stability of step-by-step methods for the numerical solution of differential equations. We refer to these three papers for more information and to \cite[Ch.II.4]{Pog} for the similar notion of order of contact between curves. 
\end{remark}

\begin{proposition}\label{prp:domain}
Let $X$ be a  Banach space and suppose that $T\in\B(X)$ is a power-bounded operator such that $W_0(T)$ is a generalised Stolz domain with parameter $\alpha\ge1$. Then $\sigma(T)\cap\T\subseteq\{1\}$ and \eqref{eq:res} holds. 
\end{proposition}
\begin{proof}
Since $\sigma(T)\subseteq W_0(T)$ by \cite[Theorem~1]{Tohoku} we have that $\sigma(T)\cap\T\subseteq\{1\}$.  It follows from \cite[Lemma~1]{Tohoku} that
\begin{equation}\label{eq:res_bd}
\|R(e^{i\theta},T)\|\le \frac{1}{\dist(e^{i\theta},W_0(T))},\quad 0<|\theta|\le\pi .
\end{equation}
By \cite[Lemma 5.1]{SpSt96} there exists a constant $c>0$ such that 
$$ \dist(e^{i\theta},W_0(T)) \ge c|\theta|^\alpha,\quad 0<|\theta|\le\pi,$$
which together with \eqref{eq:res_bd} immediately implies \eqref{eq:res}.
\end{proof}

\begin{remark}\label{rem:horo}
If $X$ is a Hilbert space then it follows from the inclusion of the numerical range in the closed unit disc that $\sup_{n\ge0}\|T^n\| \le 2$ (\cite[Ch.1,\S 11]{NaFo}), so the condition of power-boundedness holds automatically in this case. Hilbert space contractions with numerical range included in a Stolz domain were called quasi-sectorial in \cite{CaZa01}. It was proved in \cite{BaSe16} that a Hilbert space operator is a so-called unconditional Ritt operator  if and only if it is similar to an operator whose numerical range is contained in a Stolz domain.
If $r\in(0,1)$ and $\|T-rI\| \le 1-r$,  then $W_0(T)\subseteq\{\lambda\in\CC:|\lambda-r|\le1-r\}$ and in particular \eqref{eq:res} holds for $\alpha=2$. Note also that $\|T-rI\| \le 1-r$ implies that $\|T\| \le 1$. Hilbert space operators satisfying $\|T-rI\| \le 1-r$ for some $r\in(0,1)$ are characterised in \cite{Dun07, Dye89}.
\end{remark}

The next result gives a sufficient condition in the Hilbert space setting for the resolvent growth condition in Theorem~\ref{thm:gen} to be satisfied.

\begin{proposition}\label{prp:zn}
Let $X$ be a Hilbert space and let $T\in \B(X)$. Suppose that the closure $\Omega$ of the numerical range $W(T)$ of $T$
is  contained in the closed unit disc and that there exist constants $C>0$ and $\beta\in(0,1]$ such that
$$ \sup\{|\lambda^n(1-\lambda)| : \lambda \in \Omega \} \le Cn^{-\beta}, \quad n\ge 1.$$
Then $T$ is power-bounded, $\sigma(T)\cap\T\subseteq\{1\}$ and \eqref{eq:res} holds for $\alpha=1/\beta$.
\end{proposition}

\begin{proof}
Since $\Omega$ is contained in the closed unit disc,  the operator $T$ is power-bounded.  Suppose that $\lambda\in \Omega\cap \T$. Then 
$$ |\lambda-1| = |\lambda^n(\lambda-1)| \le Cn^{-\beta}$$
for all $n\ge 1$ and hence $\lambda=1$. Thus $\Omega\cap\T\subseteq\{1\}$, and since $\sigma(T)\subseteq\Omega$ we deduce that $\sigma(T)\cap\T\subseteq\{1\}$. Now let $K=1+\sqrt{2}$. By \cite{CrPa} the set $\Omega$ is a $K$-spectral set for $T$, in the sense that for every rational function $u$ with poles 
outside $\Omega$ we have 
$$ \|u(T)\| \le K\sup\{|u(\lambda)| :\lambda\in\Omega\}.$$
In particular, choosing $u(\lambda)=\lambda^n(\lambda-1)$ shows that 
$$ \|T^n(I-T)\|  \le  CKn^{-\beta},\quad n\ge1.$$
The result now follows from \cite[Theorem~3.10]{Sei16}.
\end{proof}

\begin{remark}
The above argument shows that in the setting of Proposition~\ref{prp:zn} the proof of Theorem~\ref{thm:gen} can  be simplified; see also Remark~\ref{rem:Hilbert}. Note furthermore that it is possible to use Proposition~\ref{prp:zn} to give an alternative proof of Proposition~\ref{prp:domain} for Hilbert space operators.  
\end{remark}

\section{Applications to the method of alternating projections}\label{sec:appl1}
In this section we apply the general results of the previous sections to Banach space versions of the method of alternating projections. Recall that, given a Banach space $X$,  a linear operator $P\colon X\to X$ is said to be a \emph{projection}
if $P^2=P$ or, equivalently, if $\Ker P=\Ran(I-P)$. It is clear that if $P\in\B(X)$ is a bounded projection  then either $P=0$ or
$\|P\|\ge 1$. A bounded projection $P\in\B(X)$ is said to be an \emph{orthoprojection}
if $\|P\|\le 1$.
If $X$ is a  Hilbert space then a projection $P\in\B(X)$ is an orthoprojection if and only if $\Ker P$ and $\Ran P$ are mutually orthogonal. 

\begin{remark}\label{rem:orth}
Note that  the projection $P_T$ appearing in Theorem~\ref{thm:gen} satisfies $\|P_T\|\le\sup_{n\ge0}\|T^n\|$, as can be seen from \eqref{eq:conv}. In particular, $P_T$ is an orthoprojection whenever $T$ is a contraction.
\end{remark}

Let $X$ be a Hilbert space, $N\ge2$, and suppose that $M_1,\dotsc,M_N$ are closed subspaces of $X$. For $1\le k\le N$ we let $P_k$ denote the orthogonal projection onto $M_k$, and we write $P_M$ for the orthogonal projection onto the intersection $M=M_1\cap\dotsc \cap M_N$. In many applications one wishes to study sequences in $X$ which are obtained by picking a starting vector $x\in X$ and then projecting $x$ cyclically onto the subspaces $M_1,\dots,M_N$; see for instance \cite{De92} and \cite[Chapter~9]{De01}. One is naturally led therefore to consider the operator $T\in\B(X)$ given by $T=P_N\cdots P_1$, and it is a classical result due to Halperin \cite{Hal62} that 
\begin{equation}\label{eq:Halperin}
\|T^nx-P_Mx\|\to0,\quad n\to\infty,
\end{equation}
for all $x\in X$; see \cite{BGM11, BaSe16, PuReZa13} for a discussion of  the rate of convergence in \eqref{eq:Halperin} and its dependence on the geometric relationship between the subspaces $M_1,\dotsc,M_N$. 

Our main interest here is in obtaining quantified versions of Halperin's theorem for Banach spaces with special geometric properties. Recall that a Banach space $X$ is said to be \emph{uniformly convex} if for every $\varepsilon\in (0,2]$ there exists $\delta>0$ such that for any two vectors $x,y\in X$
with $\|x\|\le 1$ and $\|y\|\le 1$ the inequality $\|x+y\|/2>1-\delta$
implies $\|x-y\|<\varepsilon$. Halperin's theorem was generalised to products of orthoprojections in uniformly convex Banach spaces by Bruck and Reich \cite{BruRei77}; see also \cite{BaLy10} and the references therein. Our aim now is to obtain a quantified result of this type, and for this we require some further terminology. Recall therefore that, given a Banach space $X$, the increasing function $\delta_X\colon[0,2]\to[0,\infty)$ given by 
$$\delta_X(\varepsilon)=\inf\left\{1-\frac{\|x+y\|}{2}:\|x\|\le 1,\|y\|\le 1,\|x-y\|\ge\varepsilon\right\}$$
is called the \emph{modulus of convexity} of the space $X$. Thus a Banach space $X$ is uniformly convex if and only if $\delta_X(\varepsilon)>0$ for all $\varepsilon\in(0,2]$. 
Any uniformly convex Banach space is reflexive and every Hilbert space is uniformly convex. Indeed, if $X$ is a Hilbert space then $\delta_X(\varepsilon)=\delta(\varepsilon)$ for all $\varepsilon\in[0,2]$, where 
$$\delta(\varepsilon) = 1-\left(1 -\frac{\varepsilon^2}{4}\right)^{1/2}\sim \frac{\varepsilon^2}{8}, \quad \varepsilon \to 0 .$$
For any Banach space $X$ we have $\delta_X(\varepsilon) \le \delta(\varepsilon)$ for all $\varepsilon\in[0,2]$; see \cite{Nor60}. For $q \ge 2$ we say that a Banach space $X$ is $q$-\emph{uniformly convex}  if there exists a constant  $c> 0$ such that $\delta_X(\varepsilon)\ge c\varepsilon^q$ for all $\varepsilon\in[0,2]$. For instance,  any $L^p$-space is $q$-uniformly convex for $q=\max\{2,p\}$ when $1<p<\infty$. According to a result of Pisier every uniformly convex Banach space can be renormed to be $q$-uniformly convex for some suitable $q\ge 2$. We refer for instance to \cite[Ch.~1.e]{LiTz} and \cite[Ch.~10]{Marting} for these results and for more information about uniformly convex Banach spaces. The first main result of this section is the following quantified version of Halperin's theorem for products of projections acting on uniformly convex Banach spaces.

\begin{theorem}\label{thm:uc}
Let $X$ be a Banach space which is $q$-uniformly convex for some $q\ge2$ and suppose that $T=P_N\cdots P_1$ for certain orthoprojections $P_1, \dots, P_N\in\B(X)$. Furthermore, let $M= \Ran P_1\cap\dotsc\cap \Ran P_N$. Then $X=M\oplus Z$, where $Z$ denotes the closure of $\Ran(I-T)$. Furthermore,  for all $x\in X$ we have
\begin{equation}\label{eq:convuc}
\|T^nx-Px\|\to0,\quad n\to\infty,
\end{equation}
where $P$ denotes the orthoprojection onto $M$ along $Z$. 
Moreover, there is a dichotomy for the rate of convergence. Indeed, if $\Ran(I-T)$ is closed then there exist $C>0$ and $r\in[0,1)$ such that
\begin{equation*}\label{eq:expuc}
\|T^n-P\|\le C r^n, \quad n\ge0,
\end{equation*}
whereas if $\Ran(I-T)$ is not closed then the convergence in \eqref{eq:convuc} is arbitrarily slow and weakly arbitrarily slow. In either case there exists a dense subspace $X_0$ of $X$ such that for all $x\in X_0$ the convergence in \eqref{eq:convuc} is superpolynomially fast.
\end{theorem}

We begin with the following lemma. 

\begin{lemma}\label{lem:halperin}
In the setting of Theorem~\ref{thm:uc} there exists $C>0$ such that
\begin{equation}
\label{eq:halperin}
\left\|x - Tx\right\| \le C \left( 1 - \|Tx\|\right)^{1/q^N} 
\end{equation}
for all $x\in X$ with $\|x\|=1$.
\end{lemma}
\begin{proof}
The proof is by induction. Suppose first that $N=1$ and that $T= P$ is an orthoprojection. 
Suppose now that $\|x\|=1$ 
and let $\ep=1-\|Px\|$. Then $\ep\in[0,1]$ and we have
\begin{equation}
\label{eq:opo}
\frac{1}{2}\|x+Px\| \ge \frac{1}{2}\|P(x+Px)\| = \|Px\| = 1-\ep.
\end{equation}
Consider the function $\beta_X \colon [0,1]\to[0,\infty)$ defined for $0\le s\le1$  by
$$\beta_X(s)=\sup\left\{\|x-y\|:\|x\|\le 1, \|y\|\le 1,\frac{\|x+y\|}{2}\ge 1-s\right\}.$$
Since $X$ is uniformly convex, we have that $\beta_X$ is continuous and increasing on $[0,1]$, with $\beta_X(0) = 0$ and $\beta_X(1) = 2$; see \cite{BHW}. We also have $\delta_X(\beta_X(s)) = s$ for all $s \in [0,1]$ by \cite[Theorem~3.4]{BHW}. It follows in particular that $\beta_X(s) \lesssim s^{1/q}$ for $0\le s\le 1$. Using \eqref{eq:opo} and the definition of $\beta_X$ we obtain
\begin{equation}\label{eq:Pest}
\left\|x - Px\right\| \le \beta_X(\ep) \lesssim \ep^{1/q} .
\end{equation}
Here and in what follows we write $a\lesssim b$ if  $a\le Cb$ for some $C>0$ which is independent of all parameters which are free to vary in the given situation. Thus \eqref{eq:halperin} is proved for $N=1$. 

Now assume that  \eqref{eq:halperin} is true for a product $S\in\B(X)$ of $N\ge 1$ orthoprojections and let $P\in\B(X)$ be a further orthoprojection. Set $T=PS$. Suppose that $\|x\|=1$ and let $\ep=1-\|Tx\| $. Using the induction hypothesis, we have
\begin{equation}\label{eq:interm}
\|x-Tx\|\le\|x-Px\|+\|P(x-Sx)\|\lesssim\|x-Px\|+\ep^{1/q^N},
\end{equation}
using the fact that $1-\|Sx\| \le  \ep$. 
 In order to estimate $\|x - Px\|$ we use the induction hypothesis to obtain
$$1-\|Px\|\le 1-\|Tx\|+\|P(x-Sx)\|\lesssim \ep  +\ep^{1/q^N} \lesssim \ep^{1/q^N},$$
so that by \eqref{eq:Pest} we have
$$ \|x-Px\| \le \beta_X(1 - \|Px\| ) \lesssim \ep^{1/q^{N+1}}. $$
Using this inequality in \eqref{eq:interm} we see that
$$\|x - Tx\| \lesssim \ep^{1/q^{N+1}} + \ep^{1/q^N} \lesssim  \ep^{1/q^{N+1}},$$
which completes the proof.
\end{proof}

Lemma~\ref{lem:halperin} leads to the following result, which is a key ingredient in the proof of Theorem~\ref{thm:uc} but is also of independent interest.

\begin{theorem}\label{thm:uc_stolz}
Let $X$ be a Banach space which is $q$-uniformly convex for some $q\ge2$ and suppose that $T=P_N\cdots P_1$ for certain orthoprojections $P_1, \dots, P_N\in\B(X)$. Then the numerical range $W_0(T)$ of $T$ is a generalised Stolz domain with parameter $q^N$. 
\end{theorem}

\begin{proof}
Fix an isometry $J\colon X\to X^*$ with the property that $\langle x, \phi_x \rangle = \|x\|^2$ for all $x\in X$. Here and subsequently we let $\phi_x:=J(x)$ for $x\in X$.
Consider an element $\lambda\in W(T)$ of the form $\lambda=\langle Tx, \phi_x \rangle$  with $\|x\| = 1$. Then $|\lambda| \le \|Tx\| \le 1$
and thus $W(T)$ is contained in the closed unit disc. By Lemma~\ref{lem:halperin} we have 
$$|\lambda-1|=|\langle Tx-x,\phi_x\rangle|\le\|x-Tx\|\lesssim (1-\|Tx\|)^{1/q^N} \le (1-|\lambda|)^{1/q^N}.$$
Using  the concavity and monotonicity properties of the functions $t\mapsto (1-t)^{1/q^N}$ and $t\mapsto t^{1/q^N}$ we see that in fact 
$$|\lambda-1|\lesssim (1-|\lambda|)^{1/q^N}$$
for all $\lambda\in W_0(T)$, the closed convex hull of $W(T)$.  It follows that $W_0(T)$ is a generalised Stolz domain with parameter $q^N$.
\end{proof}

\begin{proof}[Proof of Theorem~\ref{thm:uc}]
The result follows at once from Proposition~\ref{prp:domain} and Theorem~\ref{thm:uc_stolz}, noting that $\Fix T=M$ as a consequence of Lemma~\ref{lem:halperin}. The fact that that $P$ is an orthoprojection is clear by Remark~\ref{rem:orth}.
\end{proof}

A notion closely related to uniform convexity, and in some sense dual to it, is that of uniform smoothness. A Banach space $X$ is said to be \emph{uniformly smooth} if for every
$\ep > 0$ there exists $\delta > 0$ such that the inequality 
$$\|x+y\|+\|x-y\|<2+\ep\|y\|$$
holds for any two vectors $x,y\in X$ with $\|x\|= 1$ and $\|y\|\le\delta$.
An equivalent definition is that 
$$ \lim_{t\to0}\frac{\rho_X(t)}{t} = 0,$$
where $\rho_X$ is the \emph{modulus of smoothness} of $X$ defined for $t\ge0$ by 
$$ \rho_X(t) = \sup \left\{\frac{\|x+ty\|+\|x-ty\|}{2}-1 : x,y\in X, \|x\| = 1, \|y\| = 1\right\}.$$
A Banach space $X$ is uniformly smooth if and only if its dual $X^*$ is uniformly convex, and vice versa. In particular, any uniformly smooth space is reflexive and every Hilbert space is uniformly smooth. Indeed, if $X$ is a Hilbert space then $\rho_X(t)=\rho(t)$ for all $t\ge0$, where
$$\rho(t)=(1+t^2)^{1/2}-1\sim \frac{t^2}{2},\quad t\to0.$$
For any Banach space $X$ we have $\rho_X(t) \ge \rho(t)$ for all $t\ge0$.  For $p\in(1,2]$ we say that 
$X$ is $p$-{\em uniformly smooth} if there is a constant $C>0$ such that $\rho_X(t) \le Ct^p$ for all $t\ge 0$. For instance,  any $L^q$-space is $p$-uniformly smooth for $p=\min\{2,q\}$ when $1<q<\infty$. We refer again to \cite[Ch.~1.e]{LiTz} and \cite[Ch.~10]{Marting}  for more information.  The next result is an analogue of Theorem~\ref{thm:uc} for uniformly smooth spaces.

\begin{theorem}\label{thm:smooth}
Let $X$ be a Banach space which is $p$-uniformly smooth for some $p\in(1,2]$ and suppose that $T=P_N\cdots P_1$ for certain orthoprojections $P_1, \dots, P_N\in\B(X)$. Furthermore, let $M= \Ran P_1\cap\dotsc\cap \Ran P_N$. Then $X=M\oplus Z$, where $Z$ denotes the closure of $\Ran(I-T)$. Furthermore,  for all $x\in X$ we have
\begin{equation}\label{eq:convus}
\|T^nx-Px\|\to0,\quad n\to\infty,
\end{equation}
where $P$ denotes the orthoprojection onto $M$ along $Z$. 
Moreover, there is a dichotomy for the rate of convergence. Indeed, if $\Ran(I-T)$ is closed then there exist $C>0$ and $r\in[0,1)$ such that
\begin{equation*}\label{eq:expus}
\|T^n-P\|\le C r^n, \quad n\ge0,
\end{equation*}
whereas if $\Ran(I-T)$ is not closed then the convergence in \eqref{eq:convus} is arbitrarily slow and weakly arbitrarily slow. In either case there exists a dense subspace $X_0$ of $X$ such that for all $x\in X_0$ the convergence in \eqref{eq:convus} is superpolynomially fast.
\end{theorem}

\begin{proof}
Let $q\in[2,\infty)$ be the H\"older conjugate of $p$, so that $p^{-1}+q^{-1}= 1$. Since $X$ is uniformly smooth, its dual space $X^{\ast}$ is uniformly convex. Let
$$ \gamma_{X^{\ast}}(\ep) = \sup_{t\ge0} \left(\frac{t\ep}{2} - \rho_X(t)\right),\quad 0\le\ep\le2.$$
Then by \cite[Lemma~10.20]{Marting} we have $\delta_{X^*}(\ep)\ge\gamma_{X^*}(\ep)$ for $0\le\ep\le 2$,
and hence
$$\delta_{X^{\ast}}(\ep)\ge \sup_{t\ge0} \left(\frac{t\ep}{2} -Ct^p\right)\ge c\ep^q,\quad 0\le\ep\le2,$$
for $c=4^{-q}C^{1-q}$, as can be seen by considering $t=(\ep/4C)^{q-1}$.
Thus $X^{\ast}$ is $q$-uniformly convex. By considering the dual operator $T^*$ of $T$, which itself is a product of orthoprojections, and observing that $\|R(\lambda,T)\|=\|R(\lambda,T^*)\|$ for all $\lambda\in\CC\setminus\sigma(T)$ we deduce from Theorem~\ref{thm:uc_stolz} and Proposition~\ref{prp:domain} that condition \eqref{eq:res} of Theorem~\ref{thm:gen} is satisfied for $\alpha=q^N$. The result now follows as in the case of Theorem~\ref{thm:uc}. 
\end{proof}

We conclude this section with another application of Theorem~\ref{thm:gen} to products of orthoprojections, this time involving assumptions on the projections rather than the space. Given a Banach space $X$ and a projection $P\in\B(X)$ we shall say, in loose accordance with the terminology of \cite{BaLy10}, that $P$ is a \emph{type-D projection} if $P\ne0$ and there exists $r\in(0,1)$ such that $\|P-rI\|\le 1-r$. In particular, any type-D projection is an orthoprojection. Loosely following \cite{GKS}, a type-D projection $P\in\B(X)$ will be called a \emph{type-U projection} if $\|P-\frac12I\|\le \frac12$. For instance, any orthogonal projection on a Hilbert space is a type-U projection. Note that by Remark~\ref{rem:horo} that for a type-D projection $P$ we have $W_0(P)\subseteq\{\lambda\in\CC:|\lambda-r|\le 1-r\}$ for some $r\in(0,1)$, and hence $W_0(P)$  is a generalised Stolz domain with parameter $\alpha=2$. 

\begin{theorem}\label{thm:MAP1}
Let $X$ be a reflexive Banach space and suppose that  $T\in \B(X)$ is a convex combination of products of certain type-D projections $P_1, \dotsc, P_N$. Then the conclusions of Theorem~\ref{thm:gen} hold and the projection $P_T$ onto $\Fix T$ along the closure of $\Ran (I-T)$ is an orthoprojection. Furthermore, if 
$\Ran P_k=\{x\in X:\|P_kx\|=\|x\|\}$ for $1\le k\le N$ and if all of the $N$ projections $P_1, \dotsc, P_N$ actually appear in the decomposition of $T$, then $\Fix T=\Ran P_1\cap\dotsc\cap\Ran P_N$.
\end{theorem}

\begin{proof}
Note first that $T$, being a convex combination of products of orthoprojections, is a contraction. It was proved in \cite[Lemma~3.5]{BaLy10} that the set of contractions $Q\in\B(X)$ for which there exists $r\in (0,1)$ such that $\|Q-rI\| \le 1-r$ is a convex multiplicative semigroup. It follows in particular that $\|T-rI\|\le 1-r$ for some $r\in(0,1)$. Hence the numerical range $W_0(T)$ of $T$ is contained a generalised Stolz domain. The first part of the result now follows form Theorem~\ref{thm:gen}, and the rest is proved in  \cite{BaLy10}.
\end{proof}

\begin{remark}
If $X$ is a Hilbert space and $T=P_N\cdots P_1$, where $P_k$ is the orthogonal projection onto the closed subspace $M_k$ of $X$, $1\le k\le N$, then  $\Ran(I-T)$ is closed if and only if $M_1^{\perp}+\cdots+M_N^{\perp}$ is closed, and moreover $\Fix T=M$, where $M = M_1 \cap \dotsc \cap M_N$. By contractivity of $T$ we deduce that $P_T$ coincides with the orthogonal projection $P_M$ onto $M$. It was shown in \cite[Theorem~4.3]{BaSe16} that it is possible to obtain explicit values for the numbers $C$ and $r$ appearing in \eqref{eq:exp} in terms of the Friedrichs number of the subspaces $M_1 ,\dots,M_N $. In fact, \cite[Theorem~4.3]{BaSe16} goes beyond the present theorem in other ways too by exploiting the theory of (unconditional) Ritt operators. Indeed, it is known that in the Hilbert space case the numerical range $W(T)$ of $T$ is contained in a Stolz domain, which by means of \eqref{eq:res_bd} leads to \eqref{eq:res} with $\alpha=1$ rather than $\alpha=2$. For a closer analysis of the asymptotic behaviour in the method of alternating projections see for instance \cite{BGM11,BaSe16, PuReZa13}. 
\end{remark}

\section{The Douglas-Rachford splitting method}\label{sec:appl2}

Let $X$ be a Hilbert space and suppose that $M_1,M_2$ are closed subspaces of $X$. If for $k=1,2$ we let $P_k$ denote the orthogonal projection onto $M_k$ then we may consider the operator $T\in\B(X)$ given by 
$$T=P_2P_1+(I-P_2)(I-P_1).$$
 If we let $Q_k=2P_k-I$, $k=1,2$, then we may write $T$ as $T=\frac12(I+Q_2Q_1)$. The operator $T$ is known as the \emph{Douglas-Rachford operator} and plays an important role in the Douglas-Rachford splitting method. Here one is usually interested in the asymptotic behaviour of sequences of the form $(P_1T^nx)_{n\ge0}$ for different initial vectors $x\in X$, but in order to understand such sequences one needs first to understand the sequences $(T^nx)_{n\ge0}$ with $x\in X$; see \cite{BDNPW14, DouRac56} for further details. A generalisation to several subspaces (the \emph{cyclic} Douglas-Rachford iteration scheme) has been proposed recently in \cite{BoTa14,BoTa15}. The following result is a Banach space version of the Douglas-Rachford splitting method with several reflection operators arising from type-U projections.

\begin{theorem}\label{thm:DRban}
Let $X$ be a reflexive Banach space and let $P_1,\dotsc,P_N\in\B(X)$ be  type-U projections on $X$.
For $1\le k,\ell\le N$ let  
$$T_{k\ell}=P_kP_\ell+(I-P_k)(I-P_\ell)$$ and suppose that $T\in\B(X)$ is a convex combination of products of the operators $T_{k\ell}$,   $1\le k,\ell\le N$. Then the conclusions of Theorem~\ref{thm:gen} hold and the projection $P_T$ onto $\Fix T$ along the closure of $\Ran (I-T)$ is an orthoprojection.
\end{theorem} 
\begin{proof}
If we let $Q_k=2P_k-I$, $1\le k\le N$, then we may write $T_{k\ell}$ as $T_{k\ell}=\frac12(I+Q_kQ_\ell)$ for  $1\le k,\ell\le N$. Now 
$$\|Q_k\| = 2\Big\|P_k-\frac{1}{2}I\Big\| \le  1$$
for $1\le k\le N$ and hence 
$$ \Big\|T_{k\ell} -\frac{1}{2}I \Big\| = \frac{\|Q_kQ_\ell\|}{2} \le \frac12$$
for $1\le k,\ell\le N$. Using \cite[Lemma~3.5]{BaLy10} as in the proof of Theorem~\ref{thm:MAP1} the result now follows from Theorem~\ref{thm:gen} and Remark~\ref{rem:orth}.
\end{proof}

We conclude with the following new result on the Douglas-Rachford splitting method for the classical case of two subspaces of a Hilbert space.

\begin{theorem}\label{thm:DR}
Let $X$ be a Hilbert space and let 
$$T=P_2P_1+(I-P_2)(I-P_1)$$
 be the Douglas-Rachford operator corresponding to the orthogonal projections $P_1,P_2$ onto two closed subspaces $M_1,M_2$ of $X$.  Then for all $x\in X$ we have
\begin{equation}\label{eq:DRconv}
\|T^nx-Px\|\to0,\quad n\to\infty,
\end{equation}
where $P$ denotes the orthogonal projection onto $(M_1\cap M_1)\oplus (M_1^\perp\cap M_2^\perp)$. 
Moreover, there is a dichotomy for the rate of convergence. Indeed, if $M_1+M_2$ is closed then there exist $C>0$ and $r\in[0,1)$ such that
\begin{equation}\label{eq:DRexp}
\|T^n-P\|\le C r^n, \quad n\ge0,
\end{equation}
whereas if $M_1+M_2$ is not closed then the convergence in \eqref{eq:DRconv} is arbitrarily slow and weakly arbitrarily slow. In either case there exists a dense subspace $X_0$ of $X$ such that for all $x\in X_0$ the convergence in \eqref{eq:DRconv} is superpolynomially fast.
\end{theorem}

\begin{proof}
It is shown in \cite[Proposition~3.6]{BDNPW14} that  
$$\Fix T=(M_1\cap M_1)\oplus (M_1^\perp\cap M_2^\perp).$$
The result now follows from Theorem~\ref{thm:DRban} and the observation that \eqref{eq:DRexp} holds for some $C>0$ and $r\in[0,1)$ if and only if $M_1+M_2$ is closed; see \cite[Fact~2.3 and Theorem~4.1]{BDNPW14}. 
\end{proof}

\begin{remark}\begin{enumerate}[(a)]
\item Note that by \cite[Theorem~4.1]{BDNPW14} we may choose $C=1$ and $r=c(M_1,M_2)$ in \eqref{eq:DRexp}, where $c(M_1,M_2)$ is the Friedrichs number of the subspaces $M_1, M_2$.
\item We would like to mention that the Hilbert space results obtained in Sections~\ref{sec:appl1} and \ref{sec:appl2} easily extend to closed affine subspaces. As a final remark, we note that the methods of this paper can be applied to other projection algorithms and that, on certain classes of problems, various iterative projection methods coincide with each other. For example, if the sets are closed affine subspaces, then the method of alternating projections coincides with Dykstra's method \cite{BoDy86}. Applied to the phase retrieval problem, the method of alternating projections coincides with the error reduction method and the Douglas-Rachford method coincides with Fienup's hybrid input-output algorithm \cite{BCL}. For other such coincidences (in the cases of hyperplanes and half-spaces), see \cite{BoTa14}.
\end{enumerate}
\end{remark}


\begin{thebibliography}{99}



\bibitem{BGM11}
C.~Badea, S.~Grivaux, and V.~M\"uller, {\it The rate of convergence in the method of alternating projections},  Algebra i Analiz (St. Petersburg Math. J.), {\bf 23} (2011), 1--30.

\bibitem{BaLy10}
C.~Badea and Yu.I. Lyubich, {\it  Geometric, spectral and asymptotic properties of averaged products of projections in Banach spaces},  Studia Math. {\bf 201} (2010), 21--35. 

\bibitem{BaSe16}
C.~Badea and D.~Seifert, {\it  Ritt operators and convergence in the method of alternating
  projections}, J. Approx. Theory {\bf 205} (2016), 133--148.

  \bibitem{BHW}
J.~Bana{\'s}, A.~Hajnosz, and S.~W{e}drychowicz, {\it On
  convexity and smoothness of Banach space}, 
 Comment. Math. Univ. Carolin.
{\bf  31} (1990), 445--452.


\bibitem{BDNPW14}
H.H. Bauschke, J.Y.~Bello Cruz, T.T.A. Nghia, H.M. Phan, and X.~Wang, {\it  The rate of linear convergence of the Douglas--Rachford algorithm for
  subspaces is the cosine of the Friedrichs angle}, 
 J. Approx. Theory {\bf 185} (2014), 63--79.

\bibitem{BDNPW16}
H.H. Bauschke, J.Y.~Bello Cruz, T.T.A. Nghia, H.M. Phan, and X.~Wang, {\it Optimal rates of linear convergence of relaxed alternating projections 
and generalized Douglas-Rachford methods for two subspaces},
 Numer. Algor. {\bf73} (2016), 33--76.

\bibitem{BCL}
H. Bauschke, P.L. Combettes and D.R. Luke, {\it Phase retrieval, error reduction algorithm, and Fienup variants: a view from convex optimization},  
 J. Opt. Soc. Amer. A {\bf19} (2002), 1334--1345.


\bibitem{BoDuNR} 
F.F. Bonsall and J. Duncan, {\it  Numerical ranges of operators on normed spaces and of elements of normed algebras} and {\em Numerical ranges II},  London Mathematical Society Lecture Note Series, No. 2,  Cambridge University Press, 1971. and London Mathematical Society Lecture Notes Series, No. 10., Cambridge University Press, 1973.

\bibitem{BoTa14}
J.M. Borwein and M.K. Tam, {\it   A cyclic Douglas-Rachford iteration scheme}, 
J. Optim. Theory Appl. {\bf160} (2014), 1--29.

\bibitem{BoTa15}
J.M. Borwein and M.K. Tam, {\it   The cyclic Douglas-Rachford method for inconsistent feasibility problems}, J. Nonlinear Convex Anal. {\bf16} (2015), 573--584.

\bibitem{BoDy86} 
J.P. Boyle and R. Dykstra, {\it  A method for finding projections onto the intersection of convex sets in Hilbert spaces}. In {\it Advances in order restricted statistical inference (Iowa City, Iowa, 1985)}, pages 28--47, Lect. Notes Stat., No 37, Springer, Berlin, 1986. 

\bibitem{BruRei77}
R.E. Bruck and S. Reich, {\it  Nonexpansive projections and resolvents of accretive operators in Banach spaces}, Houston J. Math. {\bf 3} (1977), 459--470.

\bibitem{CaZa01}
V.~Cachia and V.A. Zagrebnov, {\it   Operator-norm approximation of semigroups by quasi-sectorial
  contractions}, J. Funct. Anal. {\bf180} (2001), 176--194.

 \bibitem{CoLi16} 
 G. Cohen and M. Lin, {\it  Remarks on rates of convergence of powers of contractions}, J. Math. Anal. Appl. {\bf436} (2016), 1196--1213.


\bibitem{CrPa}
M. Crouzeix and C. Palencia, {\it  The numerical range as a spectral set}, Preprint, arXiv:1702.00668, 2017.

\bibitem{De92}
F.~Deutsch, {\it The method of alternating orthogonal projections}. In {\it Approximation theory, spline functions and applications
  ({M}aratea, 1991)}, vol. 356 of {\em NATO Adv. Sci. Inst. Ser. C Math.
  Phys. Sci.}, pages 105--121. Kluwer Acad. Publ., Dordrecht, 1992.

\bibitem{De01}
F.~Deutsch, {\em Best approximation in inner product spaces}, CMS Books in Mathematics, Springer, New York, 2001.

\bibitem{DeHuSurvey1}
F.~Deutsch and H.~Hundal, {\it Arbitrarily slow convergence of sequences of linear operators: a
  survey}. In {\em Fixed-point algorithms for inverse problems in science and
  engineering}, vol.~49 of {\em Springer Optim. Appl.}, pages 213--242, Springer, New York, 2011.

\bibitem{DouRac56}
J.~Douglas and H.H. Rachford, {\it On the numerical solution of heat conduction problems in two and
  three space variables}, Trans. Amer. Math. Soc. {\bf 82} (1956), 421--439.

 \bibitem{Dun07}
N.~Dungey, {\it A class of contractions in Hilbert space and applications}, Bull. Pol. Acad. Sci. Math. {\bf55} (2007), 347--355.

\bibitem{Dye89}
J. Dye, {\it Convergence of random products of compact contractions in Hilbert space}, Integral Equations Operator Theory {\bf 12} (1989), 12--22. 


\bibitem{Est84}
J.~Esterle, {\it Mittag-{L}effler methods in the theory of {B}anach algebras and a new
  approach to {M}ichael's problem}. In {\em Proceedings of the conference on {B}anach algebras and
  several complex variables ({N}ew {H}aven, {C}onn., 1983)}, vol.~32 of {\em
  Contemp. Math.}, pages 107--129. Amer. Math. Soc., Providence, RI, 1984.
  
  \bibitem{GKS}
G.~Godefroy, N.~J. Kalton, and P.~D. Saphar, {\it Unconditional ideals in
  {B}anach spaces}, Studia Math. {\bf 104} (1993), 13--59.

\bibitem{Hal62}
I.~Halperin, {\it The product of projection operators}, Acta Sci. Math. (Szeged) {\bf 23} (1962), 96--99.

\bibitem{KT86}
Y.~Katznelson and L.~Tzafriri, {\it  On power bounded operators}, J. Funct. Anal. {\bf 68} (1986), 313--328.

\bibitem{Kre85}
U.~Krengel, {\it Ergodic Theorems},
Walter de Gruyter, Berlin, 1985.

\bibitem{LiTz}
 J.~Lindenstrauss and L.~Tzafriri, {\it Classical Banach spaces. {II}},
Ergebnisse der Mathematik und ihrer Grenzgebiete,
  vol. 97., 
  Springer-Verlag, Berlin, 1977.
 
\bibitem{Ly99}
Yu.I. Lyubich, {\it Spectral localization, power boundedness and invariant subspaces
  under {R}itt's type condition}, Studia Math. {\bf134} (1999), 153--167.

\bibitem{Mue88}
V.~M\"uller, {\it Local spectral radius formula for operators in Banach spaces}, Czechoslovak Math. J.  {\bf38} (1988), 726--729.

\bibitem{Mue05}
V.~M\"uller, {\it Power bounded operators and supercyclic vectors. II}, Proc. Amer. Math. Soc. {\bf 133} (2005), 2997--3004. 

\bibitem{NaZe99}
B.~Sz.-Nagy and J.~Zem\'anek, {\it A resolvent condition implying power boundedness}, Studia Math.  {\bf 134} (1999), 143--151.

\bibitem{NaFo}
B. Sz.-Nagy, C. Foias, H. Bercovici and L. K\'erchy, {\it Harmonic analysis of operators on Hilbert space}, 
second edition (revised and enlarged), Universitext, Springer, New York, 2010.

\bibitem{Nor60}
G.Nordlander, {\it The modulus of convexity in normed linear spaces}, 
 Ark. Mat. {\bf 4} (1960), 15--17.
 
 
\bibitem{Pau12}
V. Paulauskas, {\it A generalization of sectorial and quasi-sectorial operators}, 
 J. Funct. Anal. {\bf 262} (2912), 2074--2099.

  \bibitem{Marting}
G. Pisier, {\it Martingales in Banach spaces},
Cambridge Studies in Adv. Math. 155, Cambridge University Press, 2016. 
                   
                  
\bibitem{Pog}
A.V. Pogorelov,{\it Differential geometry}, 
translated from the first Russian ed. by L. F. Boron, P. Noordhoff N. V., Groningen, 1959.

\bibitem{PuReZa13}
E. Pustylnik, S. Reich and A.J. Zaslavski,{\it Inner inclination of subspaces and infinite products of orthogonal projections},  J. Nonlinear Convex Anal. {\bf 14} (2013), 423--436.

\bibitem{Sei15b}
D.~Seifert,{\it A quantified Tauberian theorem for sequences}, Studia Math. {\bf 227} (2015), 183--192.

\bibitem{Sei16}
D.~Seifert,{\it Rates of decay in the classical Katznelson-Tzafriri theorem},
 J. Anal. Math. {\bf130} (2016), 329--354.

\bibitem{SpSt96}
M.N. Spijker and F.A.J. Straetemans,{\it Stability estimates for families of matrices of nonuniformly bounded order}, Linear Algebra Appl. {\bf 239} (1996), 77--102.

\bibitem{SpSt97}
M.N. Spijker and F.A.J. Straetemans,{\it Error growth analysis via stability regions for discretizations of initial value problems},  BIT {\bf 37} (1997), 442--464. 

\bibitem{SpSt98}
M.N. Spijker and F.A.J. Straetemans, {\it A note on the order of contact between sets in the complex plane}, J. Math. Anal. Appl. {\bf217} (1998), 707--723.

\bibitem{Tohoku}
J.G. Stampfli and J.P. Williams, {\it Growth conditions and the numerical range in a Banach algebra},
T\^ohoku Math. J. (2) {\bf20} (1968), 417--424. 


\end{thebibliography}
\end{document}